\documentclass[11pt]{amsart}

\usepackage{verbatim, amssymb, enumitem, mathtools,color}
\usepackage{xcolor,booktabs,tabularx,array,enumitem,fancyhdr,microtype,needspace}

\newcommand{\kk}{\mathfrak k}

\newcommand \br{\mathbb{R}}

\newcommand{\dalg}{\mathfrak d}

\newcommand \Der{\operatorname{Der}}

\newcommand \codim{\operatorname{codim}}

\newcommand \Span{\operatorname{Span}}

\newcommand \cA{\mathcal{A}}

\newcommand\g{\mathfrak g}

\newcommand\h{\mathfrak h}
\newcommand\z{\mathfrak z}
\newcommand\m{\mathfrak m}
\newcommand \so{\mathfrak{so}}

\newcommand \vg{\mathfrak{v}}
\newcommand \n{\mathfrak{n}}

\newcommand \ad{\operatorname{ad}}

\DeclareMathOperator{\rank}{rank}

\newcommand \<{\langle}
\renewcommand \>{\rangle}
\newcommand \ip{\<\cdot,\cdot\>}

\newtheorem{theorem}{Theorem}
\newtheorem*{theorem*}{Theorem}
\newtheorem{corollary}{Corollary}
\newtheorem*{corollary*}{Corollary}
\newtheorem*{conj*}{Conjecture}
\newtheorem{lemma}{Lemma}
\newtheorem{proposition}{Proposition}
\newtheorem*{prop*}{Proposition}
\newtheorem*{GeoL}{Geodesic Lemma \cite{DK}}

\theoremstyle{definition}

\newtheorem*{definition*}{Definition}

\theoremstyle{remark}
\newtheorem{remark}{Remark}

\newtheorem*{notation*}{Notation}
\newtheorem*{algorithm*}{Algorithm}
\newtheorem*{example*}{Example}
\DeclareMathOperator{\ind}{ind}
\DeclareMathOperator{\rad}{rad}

\makeatletter
\@namedef{subjclassname@2020}{%
  \textup{2020} Mathematics Subject Classification}
\makeatother

\begin{document}

\title[Pseudo-Riemannian geodesic orbit nilmanifolds of signature $(n-3, 3)$]{Pseudo-Riemannian geodesic orbit nilmanifolds of signature $\boldsymbol{(n-3, 3)}$}

\author{Zhiqi Chen}
\address{School of Mathematics and Statistics, Guangdong University of Technology, Guangzhou, 510520, P.R. China}
\email{chenzhiqi@gdut.edu.cn}
\thanks{Zhiqi Chen is supported by the National Natural Science Foundation of China  (No.~12571025).}

\author{Shaoxiang Zhang}
\address{College of Mathematics and Systems Science, Shandong University of Science and Technology, Qingdao 266590, P.R. China}
\email{zhangsx@sdust.edu.cn}
\thanks{Shaoxiang Zhang  is supported by the Natural Science Foundation of Shandong Province (No.~ZR2026MS0024), Science and Technology Support Plan for Youth Innovation of Colleges and Universities of Shandong Province of China (No.~2023KJ090). }

\author{Yiyi Zhu}
\address{School of Mathematics and Statistics, Guangdong University of Technology, Guangzhou, 510520, P.R. China}
\email{yzhu51@ucsc.edu}



\begin{abstract}
We study pseudo-Riemannian $G$-geodesic orbit nilmanifolds of metric index three, where $G=N\rtimes H$, $N$ is nilpotent, and $H$ acts by isometric automorphisms. Writing $\n=\operatorname{Lie}(N)$, we prove that, if the metric is nondegenerate on the derived algebra $[\n,\n]$, then $\n$ is abelian, $2$-step, or $4$-step nilpotent. Thus the two-step rigidity known in metric indices at most two fails at index three, whereas $3$-step nilpotency remains impossible. In the $4$-step case, the derived algebra is Lorentzian, and its orthogonal
complement has index two and contains an invariant maximal totally isotropic two-plane. We construct an eight-dimensional example of signature $(5, 3)$ and prove that eight is the minimum dimension for the $4$-step case with nondegenerate derived algebra. If the metric on $[\n,\n]$ is degenerate, we establish the existence of a nonzero invariant totally isotropic subspace $F\subset\operatorname{rad}[\n,\n]$ satisfying $[F,F^\perp]=0$. The quotient $F^\perp/F$ is again geodesic orbit and has strictly smaller metric index, yielding a double-extension description of $\n$. A separate eight-dimensional example shows that this relative centrality cannot, in general, be strengthened to centrality in the full Lie algebra.
\end{abstract}

\maketitle

\section{Introduction}
\label{s:intro}

A connected pseudo-Riemannian manifold $(M,ds^2)$ is called a \emph{geodesic orbit manifold}, or a \emph{GO manifold}, if every geodesic is, up to reparametrization, an orbit of a one-parameter subgroup of its full isometry group $I(M,ds^2)$. If this property holds relative to a specified transitive Lie subgroup $G\subset I^0(M,ds^2)$, then $M$ is called a \emph{$G$-geodesic orbit manifold}, or a \emph{$G$-GO manifold}. The GO condition includes important classes of homogeneous spaces, notably symmetric, weakly symmetric, and naturally reductive spaces. For a systematic study of Riemannian GO manifold, we refer to \cite{BN}; for the reductive pseudo-Riemannian nilmanifold setting, see \cite{CNWZsurvey}.

A pseudo-Riemannian nilmanifold is a pseudo-Riemannian manifold admitting a transitive nilpotent group of isometries. The relation between such a group and the full isometry group depends substantially on the signature. For a simply connected nilpotent Lie group with a left-invariant Riemannian metric, every isometry is a composition of a left translation and an isometric automorphism \cite[Theorem~4.2]{W1963}. In indefinite signature, this description need not exhaust the full isometry group. Even for Lorentzian GO nilmanifolds, a transitive nilpotent subgroup need not be normal in the full isometry group \cite[Section~3]{dBO}. Non-reductive GO spaces also occur in dimension four \cite[Theorem~4.1]{CFZ}. Thus the choice of transitive group and reductive decomposition is an essential part of the indefinite theory.

Throughout this paper, we consider a nilpotent Lie group $N$ equipped with a left-invariant pseudo-Riemannian metric $ds^2$, and a transitive isometry group $G=N\rtimes H$, where $H$ acts on $N$ by isometric automorphisms. This gives the reductive decomposition $(N,ds^2)\cong (G/H,ds^2),  \g=\h\oplus\n$, where $\n=\operatorname{Lie}(N)$ and $\h=\operatorname{Lie}(H)$. One may take $H$ to be the maximal connected group of isometric automorphisms, although our results apply equally to smaller subgroups for which the metric is $G$-GO. We write $\ip$ for the inner product on $\n$ induced by $ds^2$, and set  $\dalg=[\n,\n],  \ip'=\ip|_{\dalg}$.
The metric index is the number of negative squares. Our principal results concern index three, equivalently signature $(n-3,3)$, where $n=\dim N$.

The algebraic form of the GO condition is provided by the geodesic lemma; see \cite{DK}.

\begin{GeoL}\label{GeoL}
Let $(M,ds^2)=G/H$ be a reductive pseudo-Riemannian homogeneous space with reductive decomposition $\g=\h\oplus\m$. Then $M$ is $G$-GO if and only if, for every $T\in\m$, there exist $A=A(T)\in\h$ and $k=k(T)\in\br$ such that
\begin{equation}\label{eq:golemma}
    \<[T+A,T']_\m,T\>
    =k\<T,T'\>
    \quad\text{for all }T'\in\m.
\end{equation}
Here $\ip$ is the inner product induced by $ds^2$, and the subscript $\m$ denotes projection onto $\m$ along $\h$.
\end{GeoL}

Taking $T'=T$ shows that $k(T)=0$ whenever $T$ is non-null. In our semidirect-product setting, $\m=\n$ and $[T+A,T']\in\n$, so the projection in \eqref{eq:golemma} may be omitted. The isotropy algebra $\h$ therefore acts on $\n$ by skew-adjoint derivations.

The nilpotency restrictions imposed by the GO condition are well understood in low metric index. Gordon proved that a Riemannian GO nilmanifold is either abelian or two-step nilpotent
\cite[Theorem~2.2]{Gor}. Under the assumption that the derived algebra is nondegenerate, the same conclusion holds in Lorentzian signature \cite[Theorem~2]{NW} and in index two
\cite[Theorem~1]{CNWZ}. The Lorentzian analysis of \cite{CWZ} allowed a possible four-step case, which was subsequently excluded in \cite{NW}. On the other hand, \cite[Section~2.3]{CNWZ} constructs a four-step GO nilmanifold of signature $(8,4)$ with nondegenerate Lorentzian derived algebra. These results leave open whether the two-step
restriction persists in metric index three.

The two-step theory itself continues to develop. Nikonorov \cite{Nik24} constructs continuous families of pairwise non-isomorphic nilpotent Lie groups admitting Riemannian GO metrics, and Nikolayevsky and Ziller \cite{NZ26} give a complete classification of non-singular Riemannian GO nilmanifolds. In indefinite signature, Furutani, Markina, and Nikonorov \cite{FMN} characterize the GO property for pseudo-$H$-type groups constructed from minimal admissible Clifford modules. Our concern here is the dependence of the nilpotency step and the extension structure on the metric index and on the degeneracy of the derived algebra.

Our first main result, Theorem~\ref{th:nondeg}, determines the possible nilpotency steps when the metric index is three and $\ip'$ is nondegenerate:
\[
    \n\text{ is abelian, $2$-step, or $4$-step nilpotent}.
\]
The four-step case can occur only when $\dalg$ is Lorentzian and its orthogonal complement has index two. In this case, the orthogonal complement contains an $\h$-invariant maximal totally isotropic two-plane. The same analysis also gives $[\dalg,\dalg]=0$. Thus metric index three is the first index at which the two-step conclusion fails under the nondegeneracy assumption, while three-step nilpotency remains impossible.

We construct an explicit eight-dimensional four-step example of signature $(5,3)$ and prove that dimension eight is minimal among four-step $G$-GO nilmanifolds of metric index three with nondegenerate derived algebra. This establishes both the optimality of the nilpotency bound and the minimum dimension in which the four-step case can occur. 

The degenerate case requires a different approach. Set $E=\operatorname{rad}\dalg =\dalg\cap\dalg^\perp$. In Theorem~\ref{thm:degenerate-complete}, we prove that there
exists a nonzero $\h$-invariant subspace $F\subset E$ satisfying $ [F,F^\perp]=0$. If $k=\dim F$, then $\m_0=F^\perp/F$ inherits a nondegenerate metric and is again a GO metric
nilpotent Lie algebra, with $\operatorname{ind}\m_0=3-k, \dim\m_0=\dim\n-2k$. Consequently, $\n$ is a $2k$-dimensional double extension of a GO metric nilpotent Lie algebra of strictly smaller index. This extends the corresponding reduction results in Lorentzian signature \cite[Theorem~3]{NW} and in index two \cite[Theorem~2]{CNWZ}.

The proof combines an invariant-kernel descent with spectral constraints on derivations and a rank analysis of the action of the degenerate radical on the nondegenerate quotient of
the derived algebra. The descent replaces a suitable invariant subspace satisfying an ideal condition by a nonzero invariant subspace that centralizes its orthogonal complement.
The remaining arguments establish the required initial condition in every configuration permitted by the index-three assumption.

Double extensions originated in the structure theory of Lie algebras with ad-invariant metrics \cite{MR}; see \cite{Ova} for a survey. Their use in the GO nilmanifold setting is developed in \cite[Section~4]{NW} and \cite[Section~3]{CNWZ}. In the present setting, the conclusion $[F,F^\perp]=0$ expresses centrality within $F^\perp$. A separate eight-dimensional GO example shows that this relative centrality need not imply centrality in the full Lie algebra.

The paper is organized as follows. Section~\ref{s:nondeg} treats the nondegenerate-derived case and proves Theorem~\ref{th:nondeg}. Subsections~\ref{ss:dl} and~\ref{sec:lorentz-derived} establish the nilpotency restrictions and the structure of the Lorentzian-derived case. Section~\ref{ss:ex} gives a further discussion on the 4-step nilpotent case, constructs the eight-dimensional example and proves its dimensional minimality. Section~\ref{sec:degenerate} develops the invariant-kernel descent and proves the double-extension reduction in Theorem~\ref{thm:degenerate-complete}.
Finally, Section~\ref{s:double} gives the example distinguishing relative centrality from global centrality.

\section{The nilpotent property when the derived algebra is nondegenerate}
\label{s:nondeg}
Given a reductive homogeneous pseudo-Riemannian manifold $(G/H,ds^2)$, where $G = N \rtimes H$, with $N$ nilpotent, we identify $\n = {\rm Lie}(N)$ with the tangent space to $G/H$ at $1N$. Let $\ip$ be the inner product on $\n$ induces by $ds^2$, and denote $\dalg = [\n,\n]$. Here the index is the number of negative squares, let $s=\ind \dalg, t=\ind \vg$.

Assume that the restriction $\ip'$ of the inner product $\ip$ to $\dalg$ is nondegenerate. Denote $\vg=(\dalg)^\perp$; note that $\n$ is the direct orthogonal sum of $\dalg$ and $\vg$, and both subspaces $\dalg$ and $\vg$ of $\n$ are $\ad_\g(\h)$-invariant. This section is to prove the following theorem.

\begin{theorem}\label{th:nondeg}
     Let $(M = G/H, ds^2)$ be a connected pseudo-Riemannian $G$-geodesic orbit nilmanifold of metric index three, where $G = N \rtimes H$, with $N$ nilpotent. Let $\ip$ denote the inner product on $\n$ induced by $ds^2$. If $\ip'=\ip|_{\dalg}$ is nondegenerate, then $N$ is abelian, two-step or four-step nilpotent.
\end{theorem}

The following is to prove Theorem~\ref{th:nondeg}. First, let $(G/H,ds^2)$ be $G$-geodesic orbit.

\begin{proposition}\label{prop:universal}
If $\dalg$ is nondegenerate, then
\begin{equation}\label{eq:skew-derived}
\<[Z,X], {X}\>=0, \quad \forall Z\in\n,\ X\in\dalg.
\end{equation}
Consequently, $\phi:\n\longrightarrow\so(\dalg)$ defined by $\phi(Z)=\ad(Z)|_{\dalg}$
is a Lie algebra homomorphism whose values are nilpotent endomorphisms. For $X\in\dalg$ and $Y\in\vg$ with $X+Y$ non-null, one may choose a single $A=A(X,Y)\in\h$ satisfying
\begin{align}
\<[A,Y], {Y'}\>&=\<[Y,Y'], {X}\>, \quad \forall Y'\in\vg,\label{eq:GO-v}\\
[A+Y,X]&=0.\label{eq:GO-d}
\end{align}
\end{proposition}
\begin{proof}
Insert $T=X+Y$ and $T'=Y$ into~\eqref{eq:golemma}, where $T$ is non-null. Orthogonality and skew-adjointness of the isotropy action give $\< [Y,X], X\>=0$. Extend this scalar identity by continuity and polarize in $X$. Thus $\ad(Y)|_{\dalg}$ is skew-adjoint for every $Y\in\vg$. Since $\vg$ generates $\n$ and commutators of skew-adjoint operators are skew-adjoint, the same holds for every $Z\in\n$. Nilpotency follows from that of $\n$, and the homomorphism property is inherited from the adjoint representation.

Finally, fix $X\in\dalg$ and $Y\in\vg$ with $X+Y$ non-null, and choose
$A\in\h$ satisfying \eqref{eq:golemma} for $T=X+Y$, again with
$k=0$. For every $Y'\in\vg$, orthogonality gives
\[
\begin{aligned}
  0
  &=\<{[X+Y+A,Y']}, {X+Y}\>=\<{[X,Y']}, {X}\>
    +\<{[Y,Y']}, {X}\>+\<{[A,Y']}, {Y}\>\\
  &=\<{[Y,Y']}, {X}\>-\<{[A,Y]}, {Y'}\>.
\end{aligned}
\]
In the last line, \eqref{eq:skew-derived}, applied with $Z=Y'$, removes
the first term, and skew-adjointness of $\ad(A)|_{\n}$ rewrites
the last one. This proves \eqref{eq:GO-v}.
For every $X'\in\dalg$, the same choice of $A$ gives
\[
\begin{aligned}
  0
  &=\<{[X+Y+A,X']}, {X+Y}\>=\<{[X,X']}, {X}\>
    +\<{[Y,X']}, {X}\>+\<{[A,X']}, {X}\>\\
  &=-\<{X'}, {[Y,X]+[A,X]}\>.
\end{aligned}
\]
Indeed, the first term vanishes by \eqref{eq:skew-derived}, applied
with $Z=X'$, while the other two terms are rewritten using the
skew-adjointness of $\phi(Y)$ and of the isotropy action.
As $[Y,X]+[A,X]\in\dalg$ and the metric on $\dalg$ is nondegenerate,
this identity implies \eqref{eq:GO-d}. Both
\eqref{eq:GO-v} and \eqref{eq:GO-d} thus hold for the same
element $A$.
\end{proof}

Put $\mathfrak k=\phi(\n)$. Engel's theorem and the triangularization used in \cite{NW,CWZ, CNWZ}, with the real-group result of \cite{Mos}, allow $\mathfrak k$ to be conjugated into the nilpotent part of an Iwasawa decomposition of $\so(\dalg)$. Then the definite case is immediate: a skew-adjoint nilpotent operator on a definite space is zero.

Therefore, to prove Theorem~\ref{th:nondeg} we need to consider three cases:
when $\ind \dalg=1, \ind \vg=2$, when $\ind \dalg=2, \ind \vg=1$, and when $\ind \dalg=3, \ind \vg=0$.

We consider these three cases separately in the following two subsections. The proof of Theorem~\ref{th:nondeg} will follow from Propositions~\ref{nondeg1} , ~\ref{nondeg2} and~\ref{nondeg3} below.

\subsection{$\ind \dalg=2, \ind \vg=1$; and $\ind \dalg=3, \ind \vg=0$}
\label{ss:dl}

In this subsection we additionally assume, in the assumptions of Theorem~\ref{th:nondeg}, that  $\ip'$
is nondegenerate, the restrictions of $\ip$ to $\vg$ is  of Lorentz signature or definite.

\begin{lemma}\label{lem:auxiliary}
Let $\mathfrak l$ be a finite-dimensional nilpotent Lie algebra generated by an invariant subspace $V$ for a Lie algebra $\mathfrak a\subset\Der(\mathfrak l)$. If $V$ admits an $\mathfrak a$-invariant positive definite inner product, then so does $\mathfrak l$.
\end{lemma}
\begin{proof}
Let $c$ be the nilpotency step. On $\mathcal T=\bigoplus_{j=1}^{c}V^{\otimes j}$,  take the direct sum of tensor-product positive definite inner products. Each element of $\mathfrak a$ acts on a tensor power as the sum of its actions on the factors, and hence acts skew-adjointly on $\mathcal T$.

Nested brackets define a surjective linear map $P:\mathcal T\longrightarrow\mathfrak l$ by
\[
P(v_1\otimes\cdots\otimes v_j)
=[v_1,[v_2,\ldots,[v_{j-1},v_j]\ldots]],
\]
with $P(v)=v$ in degree one. Surjectivity uses generation and nilpotency. The derivation identity makes $P$ equivariant. Thus $\ker P$ and its orthogonal complement in $\mathcal T$ are invariant. Transport the inner product from $(\ker P)^\perp$ through the isomorphism $P|_{(\ker P)^\perp}$ to obtain the required inner product on $\mathfrak l$.
\end{proof}

\begin{proposition}\label{nondeg1}
     Let $(M = G/H, ds^2)$ be a connected pseudo-Riemannian $G$-geodesic orbit nilmanifold where $G = N \rtimes H$ with $N$ nilpotent. If $\ip'=\ip|_{\dalg}$ is nondegenerate and $\ip|_{\vg}$ is definite, then $N$ is either abelian or $2$-step nilpotent.
\end{proposition}

\begin{proof}
The abelian case is immediate, so assume $\dalg\ne0$. Reverse the sign on the definite space $\mathfrak v$ if necessary to obtain a positive definite $\h$-invariant inner product on $\mathfrak v$. This is an auxiliary construction, not a change to the metric in the GO equations. Since $\mathfrak v$ generates $\n$, Lemma~\ref{lem:auxiliary} supplies an $\h$-invariant positive definite inner product $B$ on $\n$.

Fix $Y\in\mathfrak v$. For $X\in\dalg$ with $X+Y$ non-null, equation~\eqref{eq:GO-d} gives an $A\in\h$ satisfying $[Y,X]=-[A,X]$. Pair with $X$ using $B$:
\begin{equation}\label{eq:aux-key}
B([Y,X],X)=-B([A,X],X)=0.
\end{equation}
For fixed $Y$, the set of such $X$ is open dense in $\dalg$, because the nondegenerate quadratic form on $\dalg$ cannot be identically the constant $-\<Y, Y\>$. Therefore the polynomial identity~\eqref{eq:aux-key} holds for every $X\in\dalg$. Polarizing in $X$ makes $\phi(Y)$ skew-adjoint for the positive definite metric $B|_{\dalg}$.

On the other hand $\phi(Y)$ is nilpotent. A real skew-adjoint operator for a positive definite form is diagonalizable over $\mathbb C$ with purely imaginary eigenvalues; if nilpotent it is zero. Hence $[\mathfrak v,\dalg]=0$. The homomorphism property and generation by $\mathfrak v$ imply $\phi(\n)=0$, equivalently $[\n,\dalg]=0$.
\end{proof}


For $X\in\dalg$, define $j_X\in\so(\mathfrak v)$ by
\[
\<j_XY, {Y'}\>=\<{[Y,Y']}, {X}\>.
\]
Then equation~\eqref{eq:GO-v} is $AY=j_XY$, where we write $AY=[A,Y]$.
\begin{lemma}\label{lem:scaling}
Fix a non-null $Y\in\mathfrak v$. For every $X\in\dalg$, there exists $B\in\h$ such that
\begin{equation}\label{eq:fixed-vector}
BY=0,\quad BX=-[Y,X].
\end{equation}
\end{lemma}
\begin{proof}
Choose distinct nonzero $t_1,t_2\in\mathbb R$ for which $X+t_iY$ is non-null. Such choices exist because $\<X, X\>+t^2\<Y, Y\>$ is a nonzero polynomial. Equations~\eqref{eq:GO-v}--\eqref{eq:GO-d} give $A_i\in\h$ with
\[
A_iY=j_XY,\quad A_iX=-t_i[Y,X].
\]
Set $B=(A_2-A_1)/(t_2-t_1)$, and the conclusion follows.
\end{proof}

\begin{proposition}\label{nondeg2}
     Let $(M = G/H, ds^2)$ be a connected pseudo-Riemannian $G$-geodesic orbit nilmanifold where $G = N \rtimes H$ with $N$ nilpotent. If $\ip'=\ip|_{\dalg}$ is nondegenerate and $\ip|_{\vg}$ is Lorentzian, then $N$ is either abelian or $2$-step nilpotent.
\end{proposition}

\begin{proof}
Fix a timelike $Y\in\mathfrak v$ and let $\h_Y=\{B\in\h:BY=0\}$. The orthogonal decomposition $\mathfrak v=\mathbb RY\mathbin{\perp}(Y^\perp\cap\mathfrak v)$
is $\h_Y$-invariant. Change the sign on the negative line and retain the positive definite form on its complement. This gives an auxiliary $\h_Y$-invariant positive definite inner product on $\mathfrak v$.

Since $\mathfrak v$ generates $\n$, Lemma~\ref{lem:auxiliary} gives an $\h_Y$-invariant positive definite inner product $B_Y$ on $\n$. It depends on $Y$ and $\h_Y$, but not on $X$. For each $X\in\dalg$, Lemma~\ref{lem:scaling} supplies $B\in\h_Y$ with $BX=-[Y,X]$. Thus
\[
B_Y([Y,X],X)=-B_Y(BX,X)=0.
\]
Polarization makes $\phi(Y)$ skew-adjoint for $B_Y|_{\dalg}$. Since it is nilpotent, it is zero. This holds for every timelike $Y$.  For every $V \in \vg$, fix a timelike $Y_0\in\mathfrak v$, then \(Y_0+\varepsilon V\) remains a timelike vector, for sufficiently small nonzero \(\varepsilon\). So linearity of $\phi$ gives $\phi(\mathfrak v)=0$. Then we have $[\n,\dalg]=0$.
\end{proof}

\begin{remark}
In fact, let \((\mathfrak n,\langle\cdot,\cdot\rangle)\) be a GO metric nilpotent Lie algebra whose derived algebra \(\mathfrak d\) is nondegenerate. If the restriction of the metric to \(\mathfrak v=\mathfrak d^\perp\) is definite or Lorentzian, up to an overall reversal of sign, then \(\mathfrak n\) is at most two-step nilpotent.

\end{remark}

By Proposition \ref{nondeg1} and Proposition \ref{nondeg2}, we have
\begin{corollary}\label{cor:complement-splits}
In total index three, the nondegenerate splits $(s,t)=(2,1)$ and $(3,0)$ are at most two-step.
\end{corollary}

\subsection{$\ind \dalg=1$ and  $\ind \vg=2$}\label{sec:lorentz-derived}

Assume $\dalg$ is nondegenerate of Lorentz signature, with one negative direction. If the restriction of the metric to \(\dalg\) is negative definite, Proposition~\ref{prop:universal} immediately implies that \(\mathfrak n\) is at most two-step nilpotent. Hence we may assume that \(\dalg\) is genuinely indefinite of Lorentz signature. All orthogonal complements of subspaces of $\mathfrak v$ in this section are taken inside $\mathfrak v$, unless explicitly indicated otherwise.

Choose a Witt basis $e_1,e_2,\ldots,e_{m-1},e_m$ of $\dalg$ with
\[
\<{e_1}, {e_m}\>=1,\quad
\<{e_i}, {e_j}\>=\delta_{ij}\quad(2\leq i,j\leq m-1).
\]
Write $W_0=\Span(e_2,\ldots,e_{m-1})$. The Iwasawa nilradical in Lorentz signature is abelian. Hence there is a linear map $\Psi:\n\to W_0$ such that
\begin{equation}\label{eq:L-normal}
[T,e_1]=\Psi T,\quad [T,z]=-\<{\Psi T}, {z}\>e_m\ (z\in W_0),\quad [T,e_m]=0.
\end{equation}
Since $\phi(\dalg)=[\kk,\kk]=0$, the derived algebra is abelian and $\Psi|_{\dalg}=0$.

\begin{proposition}\label{nondeg3}
     Let $(M = G/H, ds^2)$ be a connected pseudo-Riemannian $G$-geodesic orbit nilmanifold where $G = N \rtimes H$ with $N$ nilpotent. If $\ind \dalg=1$ and  $\ind \vg=2$, then $N$ is is abelian, two-step or four-step nilpotent.
\end{proposition}

\begin{proof}
For this case, the definite subcase for $\dalg$ is immediate. Otherwise use~\eqref{eq:L-normal}. The derived algebra is abelian, and the product of any three operators $\phi(Y)$ is zero. Hence $C^5\n=0$, where $C^1\n=\n$ and $C^{j+1}\n=[\n,C^j\n]$.

If $\Psi\ne0$, choose $Y\in\mathfrak v$ with $\Psi Y\ne0$. Since $\dalg$ is spanned by brackets, some $X=[T_1,T_2]$ has nonzero $e_1$ coefficient $c$. Applying~\eqref{eq:L-normal} twice gives
\[
[Y,[Y,X]]=-c\|\Psi Y\|^2e_m\ne0.
\]
This lies in $C^4\n$, so the step is exactly four. If $\Psi=0$, the algebra is at most two-step. Thus step three is impossible.
\end{proof}

\begin{corollary}\label{cor:metabelian}
Every nondegenerate-derived GO metric nilpotent Lie algebra of total index three is metabelian, i.e.  $[\dalg,\dalg]=0$.
\end{corollary}
\begin{proof}
The three excluded splits are at most two-step. In the remaining Lorentz-derived split, the Iwasawa nilradical is abelian, so $\phi(\dalg)=0$.
\end{proof}

\section{The $4$-step nilpotent case}
\label{ss:ex}

The nilpotency alternatives and the metabelian property have now
been established. We next investigate the additional structure
of the four-step case. We therefore assume $\Psi\ne0$ in the following discussion.

For $Y_1,Y_2\in\mathfrak v$, write $[Y_1,Y_2]=\sum_{i=1}^m\omega_i(Y_1,Y_2)e_i.$
By~\eqref{eq:L-normal}, the component $e_1$ cannot be supplied by $[\mathfrak v,\dalg]$, so $\omega:=\omega_1\ne0$. The $W_0$ component of Jacobi gives
\begin{equation}\label{eq:Cartan-vector}
\omega(Y_1,Y_2)\Psi Y_3+\omega(Y_2,Y_3)\Psi Y_1+\omega(Y_3,Y_1)\Psi Y_2=0.
\end{equation}

The following lemma describes the restrictions that this identity imposes on $\Psi$ and $\omega$.

\begin{lemma}\label{lem:Cartan}
Let $0\ne\omega\in\Lambda^2 V^*$ and let $\Psi:V\to W$ satisfy~\eqref{eq:Cartan-vector}. Then $\rank\Psi\leq2$. If the rank is two, there are independent $\alpha,\beta\in V^*$ and independent $u,v\in W$ such that
\[
\Psi=\alpha\otimes u+\beta\otimes v,\quad \omega=c\alpha\wedge\beta\quad(c\ne0).
\]
If the rank is one, then $\Psi=\alpha\otimes u$ and $\omega=\alpha\wedge\gamma$ for independent $\alpha,\gamma$ and nonzero $u$.
\end{lemma}
\begin{proof}
For every $\ell\in W^*$, the one-form $\lambda=\ell\circ\Psi$ satisfies $\omega\wedge\lambda=0$. For a nonzero one-form $\lambda$, this identity makes $\omega$ divisible by $\lambda$. A nonzero two-form cannot be divisible by three independent one-forms. This proves the rank bound. Two independent such forms force $\omega$ to be a nonzero multiple of their wedge; one gives the usual Cartan factorization.
\end{proof}
Set
\begin{equation}\label{eq:CQ}
C=\ker(\Psi|_{\mathfrak v}),\quad Q=\ker\omega.
\end{equation}
Then $Q\subset C$, $\codim_{\mathfrak v}Q=2$, and $\codim_{\mathfrak v}C\in\{1,2\}$. If $\rank\Psi=2$, then $C=Q$. These conclusions are asserted only on the branch $\Psi\ne0$.

The centralizer of $\dalg$ in $\n$ is $\dalg\oplus C$, so $C$ is $\h$-invariant. Also $\mathbb Re_m$ is the radical of $[\mathfrak v,\dalg]$ by~\eqref{eq:L-normal} and is invariant. Thus $Q=\{Y\in\mathfrak v:\<{[Y,\mathfrak v]}, {e_m}\>=0\}$ and $Q^\perp$ are invariant.

\begin{theorem}
Suppose $\dalg$ is nondegenerate Lorentzian and $\ind\mathfrak v\leq2$. Either $\n$ is at most two-step nilpotent, or $\ind\mathfrak v=2$ and $L=Q^\perp\subset Q\subset\mathfrak v$ is an $\h$-invariant maximal totally isotropic two-plane, with $Q/L$ positive definite.
\end{theorem}
\begin{proof}
Suppose $\Psi\ne0$. If $C$ were nondegenerate, the nondegenerate-centralizer theorem \cite[Theorem 1(b)]{NW} would force step at most two. Hence $C$ is degenerate.

Assume $Q^\perp$ is not totally isotropic. If $\rank\Psi=1$, the degenerate hyperplane $C$ has a one-dimensional radical. If $\rank\Psi=2$, then $C=Q$ and the degenerate, non-totally-isotropic plane $C^\perp$ again has a one-dimensional radical. In both cases choose $0\ne e\in\rad C$. The line $\mathbb Re$ is $\h$-invariant.

Use~\eqref{eq:GO-v} with $Y=e$ and $Y'\in e^\perp\cap\mathfrak v$, and vary $X\in\dalg$. Since $[A,e]\in\mathbb Re$, it follows that $[e,e^\perp\cap\mathfrak v]=0$. Also $e\in C$ centralizes $\dalg$. Consequently
\begin{equation}\label{eq:e-central}
[e,e^{\perp_{\n}}]=0.
\end{equation}
Choose a null $f\in\mathfrak v$ with $\<e, f\>=1$. For every $A\in\h$, write $[A,e]=a e$. Skew-adjointness gives $[A,f]+af\in e^{\perp_{\n}}$, so the derivation identity and~\eqref{eq:e-central} show $[A,[f,e]]=0$. Applying~\eqref{eq:GO-d} to $X=[f,e]$ shows that $[f,e]$ centralizes $\mathfrak v$, and hence lies in $\z(\n)\cap\dalg$. Its $e_1$ component is zero by~\eqref{eq:L-normal}; therefore $e\in Q$.

Thus $Q^\perp=\Span(e,Y_0)$ with $\< e, {Y_0}\>=0$ and $\<{Y_0}, {Y_0}\>\ne0$. We have $\mathfrak v=\mathbb Rf\oplus\mathbb RY_0\oplus Q$ and
\[
X_0=[f,Y_0]=\sum_i\kappa_i e_i,\quad\kappa_1\ne0.
\]
Since $f\notin C$, normalize $\Psi f=e_2$. In particular $[f,e_1]=e_2$ and $[f,e_2]=-e_m$. The central vector $z=[f,e]$ has zero $e_1$ and $e_2$ components.

One further consequence of~\eqref{eq:GO-v}, with $Y=Y_0$ and $Y'\in Q$, is
\begin{equation}\label{eq:Y0Q}
[Y_0,Q]=0,
\end{equation}
because $Q^\perp$ is invariant. For any $A\in\h$, skew-adjointness also gives $[A,Y_0]=\mu e$. Choose $\lambda$ so that $X_0+f+\lambda e$ is non-null, and choose $A$ in~\eqref{eq:GO-d} for $X=X_0$, $Y=f+\lambda e$. Write $[A,f]=b f+cY_0+q$ with $q\in Q$. Then~\eqref{eq:Y0Q} and the derivation identity give
\[
0=[A+f+\lambda e,X_0]
=bX_0+\mu z+\kappa_1e_2-\kappa_2e_m.
\]
Its $e_1$ component implies $b=0$, and its $e_2$ component then implies $\kappa_1=0$, a contradiction. Hence $Q^\perp$ is totally isotropic. Its dimension is two, so $\ind\mathfrak v=2$; removing its two hyperbolic pairs leaves the positive definite quotient $Q/Q^\perp$.
\end{proof}

The following example shows that Theorem~\ref{th:nondeg} on the $4$-stepness of $\n$.  We construct an eight-dimensional example of signature
$(5,3)$ and prove that dimension eight is the minimum among four-step
$G$-GO nilmanifolds of metric index three with nondegenerate derived
algebra, under the assumption $G=N\rtimes H$, where $N$ is
nilpotent and $H$ acts by isometric automorphisms.

Let $\n$ have ordered basis $p, q, \ell_1, \ell_2, s_0, a, u, b$. Define the only nonzero brackets by
\begin{equation}\label{eq:go8-brackets}
 \begin{gathered}
 [p,q]=a,\quad [p,a]=u,\quad [p,u]=-b,\quad [p,s_0]=u,\quad [p,\ell_2]=b.
 \end{gathered}
\end{equation}
Define the nonzero metric pairings by
\begin{equation}\label{eq:go8-metric}
 \<{p}, {\ell_1}\>=\<{q}, {\ell_2}\>
 =\<{a}, {b}\>=1,
 \quad \<{s_0}, {s_0}\>=\<{u}, {u}\>=1.
\end{equation}
Then $(\n,\<{\cdot}, {\cdot}\>)$ has signature $(5,3)$,
its derived algebra is nondegenerate of signature $(2,1)$.

The subspace $\operatorname{Span}(q, \ell_1, \ell_2, s_0, a, u, b)$ is an abelian ideal of codimension one. Thus
\eqref{eq:go8-brackets} is a semidirect-product Lie bracket and satisfies
the Jacobi identity. The three hyperbolic pairs $(p,\ell_1)$,
$(q,\ell_2)$, $(a,b)$ and the two positive directions $s_0,u$ give
signature $(5,3)$. Moreover,
\begin{equation}\label{eq:go8-derived}
 \dalg=\operatorname{Span}(a,u,b),\quad
 \vg=\dalg^\perp
 =\operatorname{Span}(p,q,\ell_1,\ell_2,s_0),
\end{equation}
with signatures $(2,1)$ and $(3,2)$, respectively. The lower central
series, with $C^1\n=\n$ and
$C^{j+1}\n=[\n,C^j\n]$, is
\begin{equation}\label{eq:go8-lower-central}
 \begin{aligned}
 C^2\n=\operatorname{Span}(a,u,b), \quad
 C^3\n=\operatorname{Span}(u,b), \quad
 C^4\n=\mathbb Rb,\quad C^5\n=0.
 \end{aligned}
\end{equation}
In particular, $[p,[p,[p,q]]]=-b\ne0$.

Write  $T=(P, Q, L_1, L_2, y, x, z, \beta)$ in the ordered basis above. The bracket and metric are
\begin{equation}\label{eq:go8-coordinate-bracket}
 \begin{split}
 [T,T']=\bigl(0,0,0,0,0,\ PQ'-QP',P(x'+y')-P'(x+y),-Pz'+P'z+PL_2'-P'L_2\bigr),
 \end{split}
\end{equation}
\begin{equation}\label{eq:go8-coordinate-metric}
 \begin{split}
 \<{T}, {T'}\>={}&PL_1'+L_1P'+QL_2'+L_2Q'+yy'+x\beta'+\beta x'+zz'.
 \end{split}
\end{equation}
For $c,k,t,w\in\mathbb R$, define
\begin{equation}\label{eq:go8-A}
 \begin{split}
 \cA(c,k,t,w)T=\bigl(0,cP,-tQ-cL_2-ky,tP+wy,kP-wQ,0,-wx,wz\bigr).
 \end{split}
\end{equation}
Substitution in \eqref{eq:go8-coordinate-metric} gives
\[
 \<{\cA T}, {T'}\>+\<{T}, {\cA T'}\>=0.
\]
Thus every $\cA(c,k,t,w)$ is skew-adjoint. The derivation identity
also follows directly: both $\cA[T,T']$ and
$[\cA T,T']+[T,\cA T']$ equal
\[
 -w(PQ'-QP')u
 +w\bigl(P(x'+y')-P'(x+y)\bigr)b.
\]
Put $\mathsf H=\cA(1,0,0,0)$, $\mathsf K=\cA(0,1,0,0)$, $\mathsf Z=\cA(0,0,1,0)$, and $\mathsf W=\cA(0,0,0,1).$
Their only nonzero commutators are
\begin{equation}\label{eq:go8-h-brackets}
 [\mathsf H,\mathsf W]=\mathsf K,\quad
 [\mathsf K,\mathsf W]=-\mathsf Z.
\end{equation}
They span a Lie algebra $\h$ of skew derivations. The four
operators are simultaneously strictly triangular in the ordered basis
$p,q,a,s_0,u,\ell_2,b,\ell_1$, so the connected subgroup they generate
is a closed unipotent group of isometric automorphisms. Let $N$ be the
simply connected nilpotent group with Lie algebra $\n$ and take
the corresponding semidirect product $G=N\rtimes H$.

Define
\begin{equation}\label{eq:go8-graph}
 D(T)=\cA(x,Q+z,\beta-y,P).
\end{equation}
Let $q(T)$ be determined by
$\<{q(T)}, {T'}\>=\<{[T,T']}, {T}\>$ for every $T'$. Equations
\eqref{eq:go8-coordinate-bracket}--\eqref{eq:go8-coordinate-metric}
give
\begin{equation}\label{eq:go8-q}
 \begin{split}
 q(T)=\bigl(0,xP,-\beta Q-zy-xL_2,\beta P,Pz,0,-xP,Pz\bigr).
 \end{split}
\end{equation}
Substituting \eqref{eq:go8-graph} in \eqref{eq:go8-A} yields
$D(T)T=q(T)$. Since $D(T)$ is skew-adjoint,
\[
 \begin{aligned}
 \<{D(T)T'+[T,T']}, {T}\>
 &=-\<{T'}, {D(T)T}\>+\<{T'}, {q(T)}\>=0.
 \end{aligned}
\]
This is the geodesic lemma with $k=0$ for every $T$, including all
null vectors.

The graph is linear. Its equivariance can be checked on the four
generators using \eqref{eq:go8-h-brackets}:
\[
 \begin{array}{c|c|c}
 S &[S,D(T)]&D(ST)\\ \hline
 \mathsf H&P\mathsf K&P\mathsf K\\
 \mathsf K&-P\mathsf Z&-P\mathsf Z\\
 \mathsf Z&0&0\\
 \mathsf W&-x\mathsf K+(Q+z)\mathsf Z&-x\mathsf K+(Q+z)\mathsf Z
 \end{array}
\]
Hence $[S,D(T)]=D(ST)$ for every $S\in\h$. This gives
equivariance for the connected isotropy group, and the
natural-reductivity criterion of \cite[Proposition~2]{NW} applies.
The exact nilpotency step follows from \eqref{eq:go8-lower-central}. Therefore, it is four-step $G$-GO for a connected group of the form $G=N\rtimes H$.

\begin{remark}\label{rem:go8-rank-one}
In the notation of the Lorentzian-derived analysis, take $e_1=a$,
$W_0=\mathbb Ru$ and $e_m=b$, then we have
\[
 \Psi p=u,\quad
 \Psi q=\Psi\ell_1=\Psi\ell_2=\Psi s_0=0,
 \quad \omega=p^*\wedge q^*,
\]
where $p^*,q^*$ are coordinate covectors on $\vg$. Consequently,
\[
 \begin{aligned}
 C=\operatorname{Span}(q,\ell_1,\ell_2,s_0),\quad
 Q=\operatorname{Span}(\ell_1,\ell_2,s_0),\quad
 L=Q^\perp=\operatorname{Span}(\ell_1,\ell_2).
 \end{aligned}
\]
Thus $\operatorname{rank}\Psi=1$, $C\ne Q$, and $Q/L$ is a positive
definite line. Here $\operatorname{rad}C=\mathbb R\ell_1$ and
$C/\operatorname{rad}C$ has signature $(2,1)$. 
\end{remark}

\begin{remark}
The example of \cite[Section 2.3]{CNWZ} has Lorentzian derived algebra $\dalg$, and the orthogonal complement $\mathfrak v=\Span(f_1,\ldots,f_8)$ has signature $(5,3)$. In its displayed basis, $C=\Span(f_3,\ldots,f_8), \rad C=\Span(f_7,f_8)$,
and $C/\rad C$ has signature $(3,1)$. In contrast, the example constructed above has an orthogonal complement of index two. More precisely, in our example the invariant totally isotropic plane \(L=Q^\perp\) is maximal in \(\mathfrak v\), and the quotient \(Q/L\) is positive definite. Thus an additional negative direction in \(\mathfrak v\) is not necessary for the existence of four-step geodesic orbit nilmanifolds.
\end{remark}

\begin{lemma}
\label{lem:go8-seven-exclusion}
Under the assumptions above, there is no
seven-dimensional four-step $G$-GO metric nilpotent Lie algebra of
metric index three whose derived algebra is nondegenerate.
\end{lemma}

\begin{proof}
Suppose such an algebra exists. 
By Section~2.2, In the four-step branch we have $\operatorname{ind}\dalg=1, \operatorname{ind}\vg=2, \Psi\ne0,$
and $Q^\perp\subseteq Q$ is a totally isotropic two-plane in
$\vg$. Nonzero $\Psi$ requires $\dim\dalg\ge3$, while the
totally isotropic two-plane requires $\dim\vg\ge4$.
Dimension seven therefore forces $\dim\dalg=3, \dim\vg=4$.

Choose a Witt basis $a,u,b$ of $\dalg$ such that
$\<{a}, {b}\>=\<{u}, {u}\>=1$, with all other pairings zero.
The normal form gives a nonzero $\alpha\in\vg^*$ such that
\begin{equation}\label{eq:go8-seven-normal}
 [Y,a]=\alpha(Y)u,\quad
 [Y,u]=-\alpha(Y)b,\quad [Y,b]=0
 \quad(Y\in\vg).
\end{equation}
The derived algebra is abelian. Write
\begin{equation}\label{eq:go8-seven-forms}
 [Y_1,Y_2]=\omega(Y_1,Y_2)a
 +\eta(Y_1,Y_2)u+\nu(Y_1,Y_2)b.
\end{equation}
The rank-one conclusion of Lemma~\ref{lem:Cartan} gives $\omega=\alpha\wedge\gamma$ for a covector $\gamma$ independent of $\alpha$. The $b$-component
of the Jacobi identity for three vectors in $\vg$ gives
$\alpha\wedge\eta=0$, and hence
\begin{equation}\label{eq:go8-seven-eta}
 \eta=\alpha\wedge\delta
\end{equation}
for some $\delta\in\vg^*$.

For a covector $\theta$, $\theta(V)=\<{\theta^\sharp}, {V}\>$. This isomorphism is taken with respect to the original nondegenerate metric on $\vg$. Since
$Q=\ker\alpha\cap\ker\gamma$, the invariant-null-plane theorem
identifies
\begin{equation}\label{eq:go8-seven-L}
 L=Q^\perp
 =\operatorname{Span}(\alpha^\sharp,\gamma^\sharp)
 \subseteq Q.
\end{equation}
It is totally isotropic. As $\dim\vg=4$,
\begin{equation}\label{eq:go8-seven-selforthogonal}
 L^\perp=L.
\end{equation}

Fix $Y\in\vg$ and put $X=a+\tau u+\sigma b, \tau,\sigma\in\mathbb R$.
The vector $X+Y$ is non-null whenever $2\sigma+\tau^2+\<{Y}, {Y}\>\ne0$.

For every such pair $(\tau,\sigma)$, the identities obtained from
the geodesic lemma provide a single $A\in\h$ satisfying
\begin{equation}\label{eq:go8-seven-GO}
 [A+Y,X]=0,\quad
 \<{AY}, {Y'}\>=\<{[Y,Y']}, {X}\>,
 \quad Y'\in\vg.
\end{equation}
Every element of $\h$ preserves $\dalg$ and $\vg$.
Moreover, the characteristic line $C^4\n=\mathbb Rb$ is
invariant. Skew-adjointness therefore gives $Aa=-\rho a+\kappa u,
 Au=-\kappa b, Ab=\rho b$
for some $\rho,\kappa\in\mathbb R$.
Substitution in the first identity of \eqref{eq:go8-seven-GO}, using
\eqref{eq:go8-seven-normal}, yields
$\rho=0$ and $\kappa=-\alpha(Y)$. Thus every $A$ selected in this
way satisfies
\begin{equation}\label{eq:go8-seven-Aderived}
 Aa=-\alpha(Y)u,\quad
 Au=\alpha(Y)b,\quad Ab=0.
\end{equation}

Apply the derivation identity to $[V,a]=\alpha(V)u$, with
$V\in\vg$. Using \eqref{eq:go8-seven-Aderived} gives  $\alpha(V)\alpha(Y)b
 =\alpha(AV)u+\alpha(Y)\alpha(V)b$, and hence $\alpha\circ A=0$. Next compare the $a$-components in
the derivation identity for \eqref{eq:go8-seven-forms}. This gives $\omega(AV_1,V_2)+\omega(V_1,AV_2)=0$.

Since $\omega=\alpha\wedge\gamma$, it follows that
\begin{equation}\label{eq:go8-seven-covectors}
 \alpha\circ A=0,\quad \gamma\circ A=c\alpha
\end{equation}
for some scalar $c$.

Take $Y'=\alpha^\sharp$ in the second identity of
\eqref{eq:go8-seven-GO}. By \eqref{eq:go8-seven-covectors},
\begin{equation}\label{eq:go8-seven-alpha-central}
 \<{[Y,\alpha^\sharp]}, {a+\tau u+\sigma b}\>
 =\alpha(AY)=0.
\end{equation}
For fixed $Y$, this identity holds on an open dense subset of the
$(\tau,\sigma)$-plane, although the selected $A$ may vary with the
parameters. Its left-hand side is a linear polynomial in
$\tau,\sigma$, so it vanishes identically. Since
$[Y,\alpha^\sharp]\in\dalg$, comparison of its three coefficients
gives $[Y,\alpha^\sharp]=0$ for every $Y\in\vg$.
By \eqref{eq:go8-seven-L},
$\alpha(\alpha^\sharp)=0$. The $u$-component of this bracket,
together with \eqref{eq:go8-seven-eta}, therefore implies
\begin{equation}\label{eq:go8-seven-delta-alpha}
 \delta(\alpha^\sharp)=0.
\end{equation}

Comparing the $u$-components in the derivation identity for
\eqref{eq:go8-seven-forms}, and using
\eqref{eq:go8-seven-Aderived}--\eqref{eq:go8-seven-covectors}, gives $\alpha\wedge(\delta\circ A)
 =-\alpha(Y)\alpha\wedge\gamma$.
Consequently, for some scalar $k$,
\begin{equation}\label{eq:go8-seven-delta-action}
 \delta\circ A=-\alpha(Y)\gamma+k\alpha.
\end{equation}
Choose $Y$ with $\alpha(Y)\ne0$ and choose one admissible $X$ and
the corresponding $A$. Skew-adjointness implies $0=\<{\delta^\sharp}, {A\delta^\sharp}\>
   =\delta(A\delta^\sharp)
   =-\alpha(Y)\delta(\gamma^\sharp)
     +k\delta(\alpha^\sharp)$.
Equation~\eqref{eq:go8-seven-delta-alpha} gives
$\delta(\gamma^\sharp)=0$. Thus $\delta$ annihilates $L$, and
\eqref{eq:go8-seven-selforthogonal} implies  $\delta^\sharp\in L,
 \delta\in\operatorname{Span}(\alpha,\gamma)$.
Equation~\eqref{eq:go8-seven-covectors} now forces
$\delta\circ A\in\mathbb R\alpha$, contradicting
\eqref{eq:go8-seven-delta-action}, because
$\alpha(Y)\ne0$ and $\alpha,\gamma$ are independent.
\end{proof}

Then by the example above and Lemma~\ref{lem:go8-seven-exclusion}, we have

\begin{theorem}
\label{thm:go8-minimum}
Let $(G/H,ds^2)$ be a connected pseudo-Riemannian $G$-geodesic orbit
nilmanifold of metric index three, where $G=N\rtimes H$, $N$ is
nilpotent, and $H$ acts by isometric automorphisms. Suppose that the
restriction of the metric to $\dalg=[\n,\n]$ is
nondegenerate. If $\n$ is four-step nilpotent, then
$\dim\n\ge8.$
\end{theorem}



\section{Degenerate derived algebras: invariant-kernel descent}\label{sec:degenerate}

In this section we consider the case when the restriction of the inner product $\ip$ to the derived algebra $\dalg$ is degenerate.

Return to arbitrary ambient index, suppose $\dalg$ is degenerate, and set $E=\rad\dalg=\dalg\cap\dalg^\perp, \mathfrak v=\dalg^\perp$.
Both $E$ and $E^\perp$ are $\h$-invariant, and $\dalg\subset E^\perp$.

\begin{proposition}\label{prop:quotient}
 Let $(M = G/H, ds^2)$ be a connected pseudo-Riemannian $G$-$GO$ nilmanifold where $G =N \rtimes H$, with $N$ nilpotent. Let $0\ne F\subset E$ be $\h$-invariant and suppose
\begin{equation}\label{eq:Fideal}
[F,F^\perp]\subset F.
\end{equation}
Define the metric nilpotent Lie algebra $\mathfrak m_0=F^\perp/F$ with the inner product $\ip_0$ induced from $F^\perp$ and the pseudo-Riemannian nilmanifold $(M_0 = G_0/H_0, ds_0^2)$, where $G_0 =N_0 \rtimes H_0$, with $N_0$ the \emph{(}simply connected\emph{)} Lie group whose Lie algebra is $\m_0$, $ds_0^2$ is the left-invariant metric on $M_0$ defined by $\ip_0$, and $H_0$ is the maximal connected group of pseudo-orthogonal automorphisms of $\ip_0$.

   Then $(M_0, ds_0^2)$ is a $G_0$-$GO$ pseudo Riemannian nilmanifold. If $\n$ has signature $(p,q)$ and $k=\dim F$, the signature of $\mathfrak m_0$ is $(p-k,q-k)$.
\end{proposition}
\begin{proof}
The subspace $F^\perp$ contains $\dalg$, so it is an ideal of $\n$. Condition~\eqref{eq:Fideal} makes $F$ an ideal of $F^\perp$, which defines the quotient bracket. The radical of the metric on $F^\perp$ is exactly $F$, so the induced metric is nondegenerate and has the stated signature.

Both spaces are $\h$-invariant. Thus every $A\in\h$ induces a skew-adjoint derivation $\overline A$ of the quotient. For a lift $T\in F^\perp$ of any $\overline T\in\mathfrak m_0$, choose $A,k(T)$ in~\eqref{eq:golemma} and restrict the test vectors to $F^\perp$. Passing to the quotient gives
\[
\<{[\overline T+\overline A,\overline T']_0}, {\overline T}_0\>
=k(T)\<{\overline T}, {\overline T'}_0\>.
\]
This is the geodesic lemma for the quotient, including null vectors. Integration with the connected group of skew automorphisms gives a reductive GO nilmanifold.
\end{proof}
Centrality strengthens this conclusion: if $[F,F^\perp]=0$, then $F^\perp$ is a central extension of $\mathfrak m_0$ by $F$, and $\n/F^\perp$ is an abelian $k$-dimensional algebra. Together with the induced Witt pairing, this is the $2k$-dimensional double extension used in \cite{NW,CNWZ}. An arbitrary quotient satisfying only~\eqref{eq:Fideal} is not yet a central-extension presentation. Then we have the following theorem.

\begin{theorem}\label{thm:descent}
Suppose $0\ne F_0\subset E$ is $\h$-invariant and $[F_0,F_0^\perp]\subset F_0$. Then there is a nonzero $\h$-invariant subspace $F\subset F_0$ such that
\begin{equation}\label{eq:Fcentral}
[F,F^\perp]=0.
\end{equation}
Consequently $\n$ admits a GO double-extension reduction by $2\dim F$ dimensions.
\end{theorem}
\begin{proof}
Let $F_i\ne0$ be an invariant subspace with $[F_i,F_i^\perp]\subset F_i$, and put $I_i=F_i^\perp$. The ideal $I_i$ acts on $F_i$ by nilpotent adjoint endomorphisms. Engel's theorem gives the nonzero common kernel
\begin{equation}\label{eq:descent-sequence}
F_{i+1}=\{e\in F_i:[e,I_i]=0\}.
\end{equation}
The derivation identity shows that $F_{i+1}$ is $\h$-invariant. If $F_{i+1}=F_i$, then~\eqref{eq:Fcentral} holds for $F_i$.

Otherwise we prove that $F_{i+1}$ also satisfies the ideal condition. Fix $e\in F_{i+1}$. For every $T\in F_{i+1}^\perp$, take $T'=e$ in~\eqref{eq:golemma}. Both the $k(T)$ term and the isotropy term vanish, since $\< T, e\>=0$ and $[A,e]\in F_{i+1}$. Thus $\<{[e,T]}, T\>=0, T\in F_{i+1}^\perp$.

Polarization gives $\<{[e,T]}, U\>=-\< T, {[e,U]}\>, \quad T, U\in F_{i+1}^\perp$.
For $U\in I_i\subset F_{i+1}^\perp$, the right-hand side is zero by~\eqref{eq:descent-sequence}. Hence $[e,T]\in I_i^\perp=F_i$. Moreover, for $U\in I_i$, $[U,[e,T]]=[[U,e],T]+[e,[U,T]]=0$.
Here $[U,e]=0$, while $[U,T]\in\dalg\subset I_i$ and $e$ centralizes $I_i$. Therefore $[e,T]\in F_{i+1}$, proving $[F_{i+1},F_{i+1}^\perp]\subset F_{i+1}$.

The construction may be repeated. Each step strictly decreases the positive integer $\dim F_i$, so it terminates at a nonzero $F$ satisfying~\eqref{eq:Fcentral}. Proposition~\ref{prop:quotient} and the central-extension observation complete the proof.
\end{proof}

\subsection{Semidefinite derived algebras}

First suppose that $\ip'=\ip|_{\dalg}$  is semidefinite. Then
\begin{proposition}\label{thm:semidefinite}
Suppose the metric on $\dalg$ is semidefinite and $E=\rad\dalg\ne0$. Then
\begin{equation}\label{eq:Eideal}
[E,E^\perp]\subset E.
\end{equation}
There exists a nonzero $\h$-invariant $F\subset E$ satisfying $[F,F^\perp]=0$. Thus $\n$ is a double extension of a GO metric nilpotent Lie algebra of strictly smaller index. There is no restriction on $\dim E$ or the ambient index.
\end{proposition}
\begin{proof}
We have $E^\perp=\dalg+\mathfrak v$. Take $e\in E$, $X\in\dalg$ and $Y\in\mathfrak v$. With $T=X+Y$ and $T'=e$ in~\eqref{eq:golemma}, the isotropy and $k(T)$ terms vanish since $E$ is invariant and orthogonal to $T$. As $[e,T]\in\dalg$, which is orthogonal to $Y$, we obtain $\<{[e,X+Y]}, X\>=0$.

Setting $Y=0$ and then subtracting gives
\begin{equation}\label{eq:semidefinite-identities}
\<{[e,X]}, X\>=0,\quad\<{[e,Y]}, X\>=0.
\end{equation}
The second identity implies $[E,\mathfrak v]\subset E$. In particular $[e,E]\subset E$, because $E\subset\mathfrak v$. Thus $\ad(e)$ induces an endomorphism on $\dalg/E$. The first identity makes that endomorphism skew-adjoint for its definite induced form. It is nilpotent, and hence zero. This proves $[E,\dalg]\subset E$, and therefore~\eqref{eq:Eideal}. Apply Theorem~\ref{thm:descent} with $F_0=E$, we have $\n$ is a double extension of a GO metric nilpotent Lie algebra of strictly smaller index.
\end{proof}
The direct quotient $E^\perp/E$ is also GO by Proposition~\ref{prop:quotient}, with index reduced by $\dim E$. If $E$ is not central in $E^\perp$, it is the terminal $F$, rather than necessarily $E$, that supplies the central-extension presentation.

\Needspace{7\baselineskip}
\begin{corollary}\label{cor:radline}
Under the hypotheses of Proposition~\ref{thm:semidefinite}, if $E=\mathbb Re$, then \mbox{$[e,e^\perp]=0$}.
\end{corollary}
\begin{proof}
By~\eqref{eq:Eideal}, each $\ad(T)$, $T\in E^\perp$, preserves the line $E$. Its restriction is nilpotent and hence zero.
\end{proof}

\begin{corollary}\label{cor:semi-index3}
 Let $(M = G/H, ds^2)$ be a connected pseudo-Riemannian $G$-$GO$ nilmanifold of metric index three, where $G =N \rtimes H$, with $N$ nilpotent. Let $\ip$ denote the inner product on $\n$ induced by $ds^2$. If $\ip'=\ip|_{\dalg}$ is degenerate and semidefinite, and put $E=\operatorname{rad}\mathfrak d$. Then there exists a nonzero $\mathfrak h$-invariant subspace $F\subset E$ such that $[F,F^\perp]=0.$

Writing $k=\dim F$, we have $1\le k\le\dim E\le3,$ and the quotient $\mathfrak m_0=F^\perp/F,$ equipped with the induced metric, is a $GO$ metric nilpotent Lie algebra of index $3-k$. Consequently, $\mathfrak n$ is a double extension of $\mathfrak m_0$, with $\dim\mathfrak n=\dim\mathfrak m_0+2k.$ In particular, the reduced algebra has index two, one, or zero.
\end{corollary}

\subsection{$\ip|_{\dalg}$ has degeneracy $1$ and index $1$. }

Put $r=\dim E$ and $s=\ind(\dalg/E)$. In this case, we have $(r,s)=(1,1)$.

Apply the proof of \cite[Proposition 3, Lemmas 7--8]{CNWZ} and Proposition~\ref{prop:quotient}, we  prove the following
\begin{proposition}\label{prop:lorentz-screen}
Suppose $E=\mathbb Re$ and the nondegenerate form on $\dalg/E$ is Lorentzian, after a possible reversal of the whole metric. Then $[e, e^\perp]=0$.
There is no restriction on the signature of $\dalg^\perp/E$. Consequently $e^\perp/\mathbb Re$ is a GO quotient of index one less than the original index.
\end{proposition}

\subsection{$\ip|_{\dalg}$ has degeneracy $1$ and index $2$. } In this case, we have $(r, s)=(1,2)$. Now write $W=\dalg/E$ as a metric vector space.

\begin{lemma}\label{lem:spectrum-generation}
Let $A$ be a derivation of a finite-dimensional nilpotent Lie algebra $\n$, and put $Q=\n/\dalg$. Every complex eigenvalue of $A$ on $\n$ is a sum of eigenvalues of its induced operator on $Q$. If a Lie subalgebra $\mathfrak c$ satisfies $\mathfrak c+\dalg=\n$, then $\mathfrak c=\n$.
\end{lemma}
\begin{proof}
Write $C^1\n=\n$ and $C^{j+1}\n=[\n,C^j\n]$. Iterated brackets give an equivariant surjection $Q^{\otimes j}\to C^j\n/C^{j+1}\n$. Changing a lift by an element of $C^2\n$ changes its bracket by an element of $C^{j+1}\n$. After complexification, triangularize the single operator on $Q$; the diagonal entries of its tensor-sum action are sums of its eigenvalues. Taking quotients and then the invariant lower-central filtration proves the spectrum assertion, with no semisimplicity assumption. For the second assertion, lift a basis of $Q$ inside $\mathfrak c$. The same bracket surjections and induction along the finite lower-central filtration show that these lifts generate $\n$.
\end{proof}

For any subspace $B\ne0$ of a nilpotent algebra,
\begin{equation}\label{eq:proper-bracket-intersection}
B\not\subset[B,\n].
\end{equation}
Indeed, such an inclusion would imply $B\subset C^j\n$ for every $j$, by induction. Here and below brackets between subspaces denote their linear span.

\begin{lemma}\label{lem:inner-skew-kernel}
Let $e\in E$ and suppose $D=\ad(e)$ is skew-adjoint on all of $\n$. If $E\cap\operatorname{im}D=0$, then $e\in\z(\n)$. Moreover, for every nonzero $e\in\n$, one has $e\notin\operatorname{im}\ad(e)$.
\end{lemma}
\begin{proof}
For a skew operator, $(\ker D)^\perp=\operatorname{im}D$. Since $\operatorname{im}D\subset\dalg$,
\[
(\ker D+\dalg)^\perp=\operatorname{im}D\cap\dalg^\perp
=\operatorname{im}D\cap E=0.
\]
Thus $\ker D+\dalg=\n$. The kernel of a derivation is a subalgebra, so Lemma~\ref{lem:spectrum-generation} gives $\ker D=\n$. Finally, $[e,x]=e$ would give $\ad(x)e=-e$, contrary to nilpotency.
\end{proof}

\begin{lemma}\label{lem:fixed-small-radical}
If $1\leq\dim E\leq2$ and every $A\in\h$ fixes $E$ pointwise, then $E$ contains a nonzero central vector.
\end{lemma}
\begin{proof}
For $e\in E$, use $T'=e$ in~\eqref{eq:golemma} for non-null $T$. Since $Ae=0$ and $k(T)=0$, one obtains $\<{[e,T]}, T\>=0$. This scalar quadratic identity extends to all $T$, and polarization makes $D_e=\ad(e)$ globally skew-adjoint.

Put $J=E\cap[E,\n]$. By~\eqref{eq:proper-bracket-intersection}, $J$ is proper in $E$. If $J=0$, choose any nonzero $e\in E$. If $J$ is a line, choose a nonzero $e\in J$. In either case $E\cap\operatorname{im}D_e=0$: in the second case this intersection lies in $J$, but cannot contain $e$ by Lemma~\ref{lem:inner-skew-kernel}. Apply that lemma. The resulting line is $\h$-invariant because $\h$ fixes all of $E$.
\end{proof}

\begin{theorem}\label{thm:radical-line-complete}
Suppose $E=\mathbb Re$ and $\dalg^\perp/E$ is definite. Then $E\subset\z(\n)$. There is no restriction on the signature of $\dalg/E$. In particular, the index-three configuration $(\dim E,\ind(\dalg/E))=(1,2)$ admits reduction by the central null line $E$ to a GO metric nilpotent algebra of index two.
\end{theorem}
\begin{proof}
Fix $A\in\h$ and write $Ae=\alpha e$. The nondegenerate pairing between $\n/\dalg$ and $\dalg^\perp$ identifies their induced actions as negative duals. On $\dalg^\perp$, the invariant line $E$ has eigenvalue $\alpha$, and the quotient $\dalg^\perp/E$ has only purely imaginary eigenvalues, because its metric is definite. Thus all eigenvalues on $\n/\dalg$ have real parts $0$ or $-\alpha$.

If $\alpha>0$, Lemma~\ref{lem:spectrum-generation} makes every real part on $\n$ nonpositive, contradicting the eigenvalue $\alpha$ on $E$. If $\alpha<0$, reverse the inequalities. Hence $\alpha=0$ for every $A$. Lemma~\ref{lem:fixed-small-radical} now makes $e$ central. In the specified index-three case the index budget forces $\dalg^\perp/E$ to be positive definite (possibly zero-dimensional). Proposition~\ref{prop:quotient} supplies the asserted GO quotient.
\end{proof}

\subsection{$\ip|_{\dalg}$ has degeneracy $2$ and index $1$. } This is the most involved case. As above,  we have $(r, s)=(2, 1)$.

\begin{lemma}\label{lem:radical-action}
For every $Y\in\mathfrak v=\dalg^\perp$, $\ad(Y)$ is skew-adjoint on the possibly degenerate metric space $\dalg$ and preserves $E$. If $\dim E=2$, then $E$ is abelian, and
\[
\rho:E\longrightarrow\so(W),\quad
\rho(e)(X+E)=[e,X]+E,\quad W=\dalg/E,
\]
is a well-defined $\h$-equivariant commuting family of nilpotent skew operators.
\end{lemma}
\begin{proof}
Take $T=X\in\dalg$ and $T'=Y\in\mathfrak v$ in~\eqref{eq:golemma}. Both $\<{AY}, X\>$ and $k\< X, Y\>$ vanish by invariance and orthogonality, including when $X$ is null. Hence $\<{[Y,X]}, X\>=0$ for every $X$, and polarization gives skew-adjointness on $\dalg$. Its radical $E$ is consequently preserved. Since $E\subset\mathfrak v$, this proves that $E$ is a subalgebra. A two-dimensional nilpotent Lie algebra is abelian. Jacobi, invariance and nilpotency now give all the assertions about $\rho$; no Lie bracket on $W$ is being assumed.
\end{proof}

\begin{proposition}\label{prop:rho-rank-one}
Suppose $\dim E=2$ and $\rank\rho=1$. Then the $\h$-invariant line $F=\ker\rho$ satisfies $[F,F^\perp]=0$. This assertion does not require a signature restriction on $W$.
\end{proposition}
\begin{proof}
Choose $0\ne e\in F$ and $f\in E\setminus F$, and put $N=\rho(f)\ne0$. For $Y\in\mathfrak v$, the induced operator $\phi(Y)$ on $W$ satisfies $\rho([Y,e])=[\phi(Y),\rho(e)]=0$.
Thus $[\mathfrak v,F]\subset F$, and nilpotency on this line gives $[\mathfrak v,F]=0$.

Write $[e,a]=\lambda(a)e+\beta(a)f$ for $a\in\dalg$. Both coefficients descend to $W$ because $E$ is abelian. Jacobi applied to $e,a,X$, reduced modulo $E$, gives
\begin{equation}\label{eq:rank-one-Jacobi}
\beta(a)NX=\beta(X)Na,\quad a,X\in W.
\end{equation}
Indeed $[e,[a,X]]\in E$, and the other two terms give the displayed identity. If $\beta\ne0$, it forces $\rank N\leq1$. But a nonzero skew-adjoint operator on a nondegenerate symmetric space has positive even rank, since its associated bilinear form is alternating. Hence $\beta=0$. Then $[F,\dalg]\subset F$, and nilpotency again gives $[F,\dalg]=0$. Therefore $[F,I]=0$ for $I=E^\perp=\dalg+\mathfrak v$.

For $T\in F^\perp$, use $T'=e$ in~\eqref{eq:golemma}, its isotropy and scalar terms both vanish. Polarizing on $F^\perp$ and pairing against $I$ gives $[e,T]\perp I$, so $[e,T]\in E$. For every $X\in\dalg$, Jacobi now yields $[[e,T],X]=[e,[T,X]]-[T,[e,X]]=0$, since $e$ centralizes $\dalg$. Thus $\rho([e,T])=0$, and $[e,T]\in F$. Nilpotency on the line $F$ finally gives $[e,T]=0$.
\end{proof}

Next, We provide a proof to exclude an injective Lorentz radical action.

\begin{lemma}\label{lem:lorentz-normalizer}
Let $W$ be Lorentzian, let $\dim E=2$, and let $\rho:E\to\so(W)$ be injective with commuting nilpotent image. There are null vectors $k,\ell$ and a positive definite space $U$ such that
$W=\mathbb Rk\oplus U\oplus\mathbb R\ell, \<k, \ell\>=1$,
and, for a suitable basis $e_1,e_2$ of $E$ and orthonormal $u_1,u_2\in U$,
\begin{equation}\label{eq:plane-null-translations}
N_i=\rho(e_i),\quad N_i k=u_i,\quad
N_i x=-\<{u_i}, x\>\ell\quad  x\in U,\quad N_i\ell=0.
\end{equation}
The normalizer of $\rho(E)$ in $\so(W)$ acts conformally on the positive definite parameter plane $\Span(u_1,u_2)$. Its conformal scalar equals its eigenvalue on $\mathbb R\ell$.
\end{lemma}
\begin{proof}
A nonzero nilpotent Lorentz-skew operator $N$ has a three-step Jordan chain. Indeed, a chain of length at least four would have two linearly independent, mutually orthogonal null vectors at its end, impossible in Lorentz signature. A square-zero skew operator has totally isotropic image, hence rank at most one, and therefore is zero by evenness of skew rank. Thus $N^3=0$ and $\operatorname{im}N^2$ is nonzero and totally isotropic, so it is a null line $L$.

Every commuting operator preserves $L$. A nilpotent skew operator preserving $L$ acts trivially on $L$ and on the definite quotient $L^\perp/L$; it therefore has precisely the null-translation form in~\eqref{eq:plane-null-translations}, of rank two when nonzero. The translation parameters span a positive plane and can be orthonormalized by changing the basis of $E$.

The line $L$ is the radical of the common kernel, so every normalizing skew operator $K$ preserves it. Write $K\ell=\delta\ell$; its induced action on $U$ is a skew operator $B$. Directly from the Witt blocks, $[K,N_u]=N_{Bu+\delta u}$. On an invariant parameter plane this is a skew operator plus $\delta$ times the identity, as asserted.
\end{proof}

\begin{proposition}\label{prop:rho-not-injective}
If $\dim E=2$, $W=\dalg/E$ is Lorentzian and $\mathfrak v/E$ is definite, then $\rank\rho<2$.
\end{proposition}
\begin{proof}
Suppose $\rho$ is injective. Equip $E$ with the positive metric transported from the translation parameters in Lemma~\ref{lem:lorentz-normalizer}. For $A\in\h$, we have $A|E$ is conformal, with a scalar $\delta$. Its eigenvalues have real part $\delta$. The generator quotient $\n/\dalg$, dual to $\mathfrak v$, has eigenvalue real parts $-\delta$ or zero. Lemma~\ref{lem:spectrum-generation} therefore forces $\delta=0$, just as in Theorem~\ref{thm:radical-line-complete}. Thus $\h|E\subset\so(E)$, and the induced action of $\h$ on the null line $\mathbb R\ell\subset W$ is zero.

We also have $[\mathfrak v,E]=0$. Indeed $\phi(Y)=\ad(Y)|W$ for $Y\in\mathfrak v$ normalizes $\rho(E)$, so its action on $E$ is conformal for the same positive metric. That action is nilpotent. Its scalar is zero by trace, and its remaining skew part is zero by definiteness.

If $\h$ acts trivially on $E$, Lemma~\ref{lem:fixed-small-radical} gives a nonzero central vector in $E$, contradicting injectivity of $\rho$. We may therefore assume that its action contains a nonzero rotation and is irreducible on the real plane $E$.

Put $\mathcal S=[E,\n]$ and $J=E\cap\mathcal S$. These spaces are $\h$-invariant and $J$ is proper by~\eqref{eq:proper-bracket-intersection}. Irreducibility gives
\begin{equation}\label{eq:plane-disjoint-images}
E\cap[E,\n]=0.
\end{equation}
Choose a complement $W_0$ to $E$ in $\dalg$ containing $\mathcal S$. It is nondegenerate, and identifies isometrically with $W$. Choose a Witt decomposition $\n=\Span(P_1,P_2)\oplus W_0\oplus V\oplus E,  \<{P_i}, {e_j}\>=\delta_{ij}$,
where $V\oplus E=\mathfrak v$, both $P_i$ are null and orthogonal to $W_0\oplus V$, and $V$ is definite. These complements are not assumed $\h$-invariant. Let $D_i=\ad(e_i)$ on $\n$. Its entire image lies in $\mathcal S\subset W_0$, so its restriction to $W_0$ is exactly $N_i$; the vectors in~\eqref{eq:plane-null-translations} will now denote their lifts to $W_0$.

For $T\in\n$, put $p_i(T)=\<{e_i}, T\>$ and $q_i(T)=\<{D_iT}, T\>$. In~\eqref{eq:golemma}, take the test vector $p_1e_1+p_2e_2$. For non-null $T$, its isotropy pairing is zero because $A|E$ is skew for the auxiliary positive metric. Thus $p_1q_1+p_2q_2=0$. This scalar polynomial identity holds for all $T$. Since $p_1,p_2$ are independent linear forms, polynomial divisibility gives a linear form $a$ with
\begin{equation}\label{eq:plane-cubic-defect}
q_1=a p_2,\quad q_2=-a p_1.
\end{equation}
Write $b=a^\sharp$ for its metric dual. Polarization yields
\begin{equation}\label{eq:plane-symmetric-defects}
D_1+D_1^*=e_2\mathbin{\odot}b,\quad
D_2+D_2^*=-e_1\mathbin{\odot}b,
\end{equation}
where $(x\odot y)T=\< y, T\> x+\<x, T\> y$.

Every $D_i$ annihilates $\mathfrak v$, and $D_i^*\mathfrak v=0$ because $\operatorname{im}D_i\subset\dalg$. Equation~\eqref{eq:plane-symmetric-defects} gives $b\perp\mathfrak v$, so $b\in\dalg$. Furthermore $D_iP_j\in W_0\perp\Span(P_1,P_2)$. Evaluate~\eqref{eq:plane-cubic-defect} at $P_2$ for $i=1$ and at $P_1$ for $i=2$, $\< b, {P_1}\>=\<b, {P_2}\>=0$. Thus $b\in W_0$. Polarizing against $W_0$ now gives the exact identities
\begin{equation}\label{eq:plane-P-actions}
D_1P_1=0,\quad D_1P_2=b,\quad
D_2P_1=-b,\quad D_2P_2=0.
\end{equation}

The vector $b$ is fixed by $\h$. For a direct verification, let $\mathcal J e_1=e_2$, $\mathcal J e_2=-e_1$. The symmetric defect of $D_e$ is $(\mathcal J e)\odot b$. Commuting this identity with any skew derivation $A\in\h$, using $[A,D_e]=D_{Ae}$ and $A|E$ commuting with $\mathcal J$, gives $(\mathcal J e)\odot Ab=0$, hence $Ab=0$. Applying the geodesic lemma to the fixed test vector $b$, first on non-null $T$ and then polarizing, proves that $\ad(b)$ is globally skew. It consequently annihilates $\mathfrak v$, and in particular $[b,E]=0$.

The vector $\ell=-D_iu_i$ lies in $\mathcal S$. Its class in $W$ is fixed by $\h$, since the normalizer scalar was shown to be zero. Therefore $A\ell\in E\cap\mathcal S=0$, $\ell$ is fixed as an actual vector, without requiring an invariant complement. The same GO argument makes $\ad(\ell)$ globally skew, so $[\ell,\mathfrak v]=0$.
Now
\begin{equation}\label{eq:plane-b-ui}
[b,u_i]=[b,D_i k]=D_i[b,k]\in\Span(u_i,\ell).
\end{equation}
For each $j$ and each $i=1, 2$, Jacobi, $\ell=-D_i u_i$, and~\eqref{eq:plane-P-actions} give $[P_j,\ell]=-[[P_j,e_i],u_i]-D_i[P_j,u_i]
\in\Span(u_i,\ell)$.
Here the first term is zero or a signed multiple of~\eqref{eq:plane-b-ui}, the second lies in the same plane because $[P_j,u_i]\in\dalg$. The intersection of the two planes is $\mathbb R\ell$. Nilpotency on this line yields $[P_j,\ell]=0$. Since $P_1,P_2,\mathfrak v$ together generate $\n$, we conclude that $\ell\in\z(\n)$.

Set $c=[P_1,P_2]$ and write its $W_0$-component as $c_0=\tau k+x+\sigma\ell$, with $x\in U$. Jacobi and~\eqref{eq:plane-P-actions} give $[P_1,b]=N_1c_0, [P_2,b]=N_2c_0$.
By~\eqref{eq:plane-b-ui}, $[b,u_i]\in W_0$, which is orthogonal to $P_j$. Skew-adjointness of $\ad(b)$ therefore gives
\[
0=\<{[b,u_i]}, {P_j}\>=-\<{u_i}, {[b,P_j]}\>
=\<{u_i}, {N_jc_0}\>=\tau\delta_{ij}.
\]
Hence $\tau=0$ and $[P_i,b]\in\mathbb R\ell$. Together with $[b,\mathfrak v]=0$ and $\ell\in\z(\n)$, this means that $\ad(b)$ sends every generator to a central vector. It annihilates $\dalg$, and its image is contained in $\mathbb R\ell$. Global skew-adjointness also puts its image in $\dalg^\perp$. Since $\mathbb R\ell\subset W_0$ and $W_0\cap\dalg^\perp=0$, we get $\ad(b)=0$.

Finally,~\eqref{eq:plane-P-actions} and $D_i\mathfrak v=0$ show that $D_i$ sends every generator into the central subspace $\mathbb Rb$ (which may be zero). Thus $D_i\dalg=0$, contradicting the nonzero operators $N_i$. This excludes rank two.
\end{proof}

Now we have
\begin{theorem}\label{thm:radical-plane-complete}
Suppose $\dim E=2$, $\dalg/E$ is Lorentzian and $\dalg^\perp/E$ is definite. Then $\rank\rho\leq1$, and there is a nonzero $\h$-invariant $F\subset E$ with $[F,F^\perp]=0$. In total index three the quotient $F^\perp/F$ has index $3-\dim F\in\{1,2\}$ and is GO. If $\rank\rho=1$, one can take $F=\ker\rho$, of dimension one.
\end{theorem}

\begin{proof} 
Proposition~\ref{prop:rho-not-injective}  implies that \(\operatorname{rank}\rho\leq 1\). If \(\operatorname{rank}\rho=1\), Proposition~\ref{prop:rho-rank-one} gives the required invariant line \(F=\ker\rho\). If \(\rho=0\), then \([E,\mathfrak d]\subseteq E\), while Lemma~\ref{lem:radical-action} yields \([E,\mathfrak v]\subseteq E\). Since \(E^\perp=\mathfrak d+\mathfrak v\), we obtain \([E,E^\perp]\subseteq E\). Theorem~\ref{thm:descent} therefore provides a nonzero invariant subspace \(F\subseteq E\) satisfying \([F,F^\perp]=0\). The remaining assertions follow from Proposition~\ref{prop:quotient} . 
\end{proof}

Then we have the following theorem.

\begin{theorem}\label{thm:degenerate-complete}
Let $(M = G/H, ds^2)$ be a connected pseudo-Riemannian $G$-geodesic orbit nilmanifold with metric index three, where $G =N \rtimes H$, with $N$ nilpotent. Let $\ip$ denote the inner product on $\n$ induced by $ds^2$. If $\ip'=\ip|_{\dalg}$ is degenerate, there admits a nonzero $\h$-invariant subspace $F\subset\rad\dalg$ satisfying $[F,F^\perp]=0$. Then $(\n, \ip)$ is a $2\dim F$--dimensional double extension of a metric Lie algebra of strictly smaller index corresponding to a $GO$ nilmanifold.
\end{theorem}

\begin{proof}
Put $r=\dim E$ and $s=\ind(\dalg/E)$. Semidefinite restrictions are covered by Proposition~\ref{thm:semidefinite}. In the indefinite case, $r+s\leq3$ leaves $(r,s)=(1,1),(1,2),(2,1)$. Proposition~\ref{prop:lorentz-screen} covers $(1,1)$. For the latter two cases, $\mathfrak v/E$ is positive definite by the index budget. Apply Theorems~\ref{thm:radical-line-complete} and~\ref{thm:radical-plane-complete}, respectively.
\end{proof}

\begin{remark}
 Under the hypotheses of the theorem, repeated application of the lower-index structure theorems yields a $GO$ metric nilpotent Lie algebra $\mathfrak m_0$ of index $q_0 \in \{0,1,2\}$, which is at most two-step nilpotent, such that $\mathfrak n$ is obtained from $\mathfrak m_0$ by a finite sequence of double extensions. The total dimension added is
$$\dim\mathfrak n-\dim\mathfrak m_0=2(3-q_0).$$
Thus the total extension dimension is $2$, $4$, or $6$, according as $\mathfrak m_0$ has index two, one, or zero.
\end{remark}

\section{Example: centrality cannot be strengthened indiscriminately}
\label{s:double}

In this section we give an eight-dimensional example to show why centrality in the orthogonal complement cannot generally be strengthened to global centrality.

On the basis $p,q,a,e,f,u,v,w$ take the nonzero brackets
\[
[p,q]=a,\quad[p,a]=e,\quad[q,a]=f,\quad
[p,e]=u,\quad[p,f]=[q,e]=v,\quad[q,f]=w.
\]
This is the free nilpotent algebra of rank two and step four. Give it the metric
\[
\< p, f\>=1,\quad\< q, e\>=-1,\quad\< a, a\>=\< v, v\>=1,\quad\< u, w\>=-2,
\]
with all other pairings zero. Its signature is $(5,3)$,
\[
\dalg=\Span(a,e,f,u,v,w),\quad E=\Span(e,f),\quad
\operatorname{signature}(\dalg/E)=(3,1).
\]
Here $E^\perp=\dalg$ and $[E,E^\perp]=0$, but $[p,e]=u\ne0$. Thus $E$ need not be central in all of $\n$.

For completeness, an exact GO certificate is as follows. For $Z=xu+yv+zw$, let the derivation $D_Z$ act on both $\Span(p,q)$ and $\Span(e,f)$ by
\[
K_Z=\begin{pmatrix}y&-2x\\2z&-y\end{pmatrix},
\]
annihilate $a$, and act on $\Span(u,v,w)$ by the symmetric-square representation. These are skew derivations. For a vector $T=P+A a+Q+Z$, with $P\in\Span(p,q)$ and $Q\in\Span(e,f)$, use $D(T)=D_Z$.

Identify the two planes with $\mathbb R^2$, with alternating form $\omega(p,q)=1$, so their metric pairing is $\omega(P,Q)$. Write $S(P,Q)=[P,Q]\in\Span(u,v,w)$ and let $B$ be the metric on that last space. Directly from the brackets, the $a$-terms cancel and
\[
\< T, {[T,T']}\>=B\bigl(Z,S(P,Q')-S(P',Q)\bigr).
\]
The matrix above satisfies $\omega(P,K_ZQ)=-B(Z,S(P,Q))$ and $D_ZZ=0$. Skewness then gives $\< T, {D_ZT'+[T,T']}\>=0$ for every $T,T'$, including null $T$. The graph is also equivariant under the displayed $\mathfrak{sl}_2(\mathbb R)$ derivations.

In this example, the isotropy action on $E$ is noncompact, whereas $\rho=0$ because $[E,\dalg]=0$. There is therefore no conflict with the compact action obtained in the proof of Proposition~\ref{prop:rho-not-injective}, where compactness follows under the temporary assumption that $\rho$ is injective. The example confirms that the relative-centrality conclusion of the reduction theorem has the appropriate scope.


\end{document}